\documentclass[11pt]{article}
\usepackage{graphicx}
\usepackage[english]{babel}

\usepackage[utf8]{inputenc}
\usepackage{amsmath,amsfonts,amssymb,amscd,amsthm,txfonts,hyperref}
\usepackage[all]{xy}
\usepackage{multibib}

\usepackage{newunicodechar}

\newunicodechar{≥}{\ensuremath{\geq}}

\newtheorem{theorem}{Theorem}[section]
\newtheorem{proposition}{Proposition}[section]
\newtheorem{lemma}{Lemma}[section]
\newtheorem{corollary}{Corollary}[section]
\newtheorem{remark}{Remark}[section]

\theoremstyle{definition}
\newtheorem{definition}{Definition}[section]

\newtheorem{notation}{Notation}[section]

\newcommand{\T}{\varmathbb T}
\newcommand{\C}{\varmathbb C}
\newcommand{\R}{\varmathbb R}
\renewcommand{\S}{\varmathbb S}
\newcommand{\Z}{\varmathbb Z}

\newcommand{\zon}{\mathrm{zon}}

\newcommand{\N}{\varmathbb N}

\newcommand{\E}{\varmathbb E}

\newcommand{\an}[1]{\langle #1 \rangle}

\title{Construction of a Gibbs measure for the zonal Dirac equation}
\author{Anne-Sophie de Suzzoni\footnote{LaMME, université d'\'Evry Paris Saclay, UMR 8071}, Cyril Malézé\footnote{CMLS, école polytechnique, UMR 7640}}
\date{}

\begin{document}

\maketitle

\begin{abstract} 
We propose a framework to construct Gibbs measures for the Dirac equation. We consider the Dirac equation on the sphere with a "Hartree-type" nonlinearity. We consider a zonal model, that is the analog of a spherically symmetric model but on the sphere. We build a Gibbs measure for this model. With a compactness argument, we prove the existence of a random variable that is a weak solution to the Dirac equation and whose law is the Gibbs measure at all times.
\end{abstract}

\section{Introduction}

\subsection{The zonal Dirac equation and Gibbs measures}

We consider the following equation 
\begin{equation}\label{Dirac1} 
i\partial_t u = iD_d u + \an{u,u}_{\C^{2^{\lfloor \frac{d}{2} \rfloor}}} u
\end{equation}
where $D_d$ is a Dirac operator on a manifold of dimension $d$. The issue at stake is to build Gibbs measures for nonlinear Dirac equations and prove their invariance under the flow of said Dirac equations. For this purpose, we follow the method presented by Oh and Thomann in \cite{oh2018pedestrian}.

The question of the invariance of Gibbs measures for Hamiltonian PDEs has been widely investigated in the literature. It took its origine in the work by Lebowitz Rose and Speers \cite{lebowitz1988statistical}. One of the first works on this matter was done by Bourgain in \cite{bourgain1994periodic,bourgain1996invariant} for the nonlinear Schrödinger equation and in \cite{bourgain1997invariant} for the Gross-Pitaevskii equation. The work of Bourgain has been continued in many other articles. The question of Gibbs measure for the nonlinear Schrödinger has been the question of many works, for instance, one can cite \cite{burq2013long,lebowitz1988statistical,mckean1995statistical,nahmod2012invariant,thomann2010gibbs,tzvetkov2006invariant,tzvetkov2008invariant,zhidkov2003korteweg}. We also mention the works of Burq and Tzvetkov on supercritical wave equations \cite{burq2007invariant,burq2008random,burq2008random2}. Finally, several works on many other equations can be found in \cite{chapouto2024deep,de2014wave,deng2015invariance,deng2024invariant,gassot2022probabilistic,mckean1999statistical,oh2009invariant,oh2010invariance,oh2012interpolation,tzvetkov2010construction}.  We mention with a particular emphasis works about the Gibbs measure in 3d for nonlinear wave equation \cite{bringmann2024invariant} and on the Schrödinger equation on the 2 dimensional sphere \cite{burq2024probabilistic}, which seem to approach some type of criticality in these matters, namely, both works are very close to to a result that cannot be improved in terms of dimension or, if one prefers, of regularity.

Historically, the introduction of Gibbs measure allowed to access an invariant to a given Hamiltonian equation whose regularity is below the energy space. This permits to prove global well-posedness below the regularity of the energy. It was later used to improve results of  well-posedness results for regularity below the critical one : by introducing Gibbs measures, we consider random initial datum with Fourier coefficients randomly distributed and expect to have better well-posedness results than in the deterministic case. Gibbs measures allow, in this sense, to introduce a particular case of problems with random initial datum.    

Although there exists some literature on Gibbs measures on noncompact domains \cite{bourgain2000invariant,bringmann2025invariant,mckean1999statistical,cacciafesta2015invariant,cacciafesta2020invariance}, it remains relatively scarce and is often restricted to the real line. The literature for Gibbs measures on compact manifolds is of a greater extent, and the results are somewhat stronger. Therefore, noncompact domains might not be such a natural starting point for studying Gibbs measures for Dirac equations and thus we focus on compact manifolds. For many reasons, we have chosen to work on the sphere. One of them is that the spectral theory of the Dirac operator on the sphere is well known. It is also boundary-free which avoids technicalities regarding physically relevant boundary conditions. But the main reason is that it has a lot of symmetries that one can use to reduce the dimension. The bigger the dimension is, the more difficult it becomes to build the measure and this becomes crucial for Dirac equations. 

Even on the sphere, which is a relatively nice manifold, there are many problems to build a Gibbs measure for \eqref{Dirac1}. The first one we raise is the order of $D_d$. It is an operator of order 1 and a good analog for equation \eqref{Dirac1} is the semiwave equation on the torus
\[
i\partial_t u = \sqrt{-\Delta} u + |u|^2u.
\]
Even in dimension 1, one needs to renormalize the equation to make sense of the Gibbs measure. Indeed, the Gaussian field on which is based the Gibbs measure, namely
\[
\sum_{n \in \Z} \frac1{\sqrt{1+|n|}} e^{inx} g_n
\]
where $(g_n)_n$ is an iid sequence of centered and normalised Gaussian variables, has values in $H^{-s}(\varmathbb T)$ for any $s>0$ but almost surely not $L^2(\T)$. Hence one cannot make sense of the product $|u|^2 u$ and needs to renormalise to give it a sense. In higher dimension, the regularity is even lower. The analog in the literature would be the Schrödinger equation in dimension 3, for which the invariance of the Gibbs measure remains unsolved to our best knowledge. Hence, we look for a model that is in dimension 1. Unfortunately, the Dirac equation on $\varmathbb S^1$ is not very relevant, both physically and mathematically. To bypass this issue, we use the symmetries of the sphere and work on zonal symmetry. On $\R^3$, there is a radial model for the Dirac equation which is called (somewhat abusevively) the Soler model. It uses the fact that the Dirac operator enjoys a so-called \emph{partial wave decomposition} which means mainly that the diagonalisation of the Dirac operator on the sphere $\S^2$ may be computed thanks to spherical harmonics, that $\R^3$ might be seen as a warped product between $\R^+$ and $\S^2$ and that the action of the Dirac operator may be decomposed between its action on the radial variable and the angular variable. In this context the lowest spherical harmonics components are stable under the action of the Dirac operator, but they also enjoy that  
\[
\an{u,u}_{\C^{2^{\lfloor \frac{d}{2} \rfloor}}},
\]
depends only on the radial variable, where $\an{\cdot,\cdot}_{\C^{2^{\lfloor \frac{d}{2} \rfloor}}}$ is the canonical scalar product on $\C^{2^{\lfloor \frac{d}{2} \rfloor}}$. Therefore, the nonlinear Dirac equation \emph{a priori} preserves the lowest spherical harmonics. But $\S^d$ may also be seen as the warped product between $(0,\pi)$ and $\S^{d-1}$ and thus we have a similar partial wave decomposition. We call this model \emph{zonal} rather than radial. We give a proper definition in Section \ref{sec:defzonDirac}. 

As we have already mentioned, even in dimension 1, one needs to renormalise the equation to give it a sense; this transforms equation \eqref{Dirac1} into
\begin{equation}\label{Dirac2}
i\partial_t u = iD_d u + :\an{u,u}_{\C^{2^{\lfloor \frac{d}{2} \rfloor}}} u :
\end{equation}
where $:\an{u,u}_{\C^{2^{\lfloor \frac{d}{2} \rfloor}}} u :$ is the distribution limit of $(\an{\Pi_N u, \Pi_N u}_{\C^{2^{\lfloor \frac{d}{2} \rfloor}}} - \sigma_N )\Pi_N u $ where $\Pi_N$ projects onto spherical harmonics of degree less than $N$ and $\sigma_N$ is a well-chosen function, see Subsection \ref{subsec:renormalizednonlin}.

Equation \eqref{Dirac2} restricted to zonal functions presents some analogies with the Schrödinger equation on the sphere $\S^2$ that we now comment. Burq, Camps, Sun and Tzvetkov, \cite{burq2024probabilistic} have developed a strategy that allows to obtain results on an equation with a local nonlinearity (such as we have presented so far). But on the sphere, they use a finely tuned basis (with an argument inspired by \cite{burq2013injections}) of eigenspaces of the Laplace-Beltrami operator. This in a broad sense allows to improve the $L^2 \rightarrow L^p$ injection norm satisfied by the basis. As we know, spherical harmonics saturate this norm. We also mention that Gibbs measures on the sphere seems harder to handle than their analog on the torus, see \cite{burq2025second}. Unfortunately, in our case, the zonal symmetry prescribes the basis and, even worse, our basis saturates the $L^2 \rightarrow L^p$ injection norm. To circumvent this issue, we use a nonlocal nonlinearity. Soler models include Hartree-type nonlinearities, that is 
\[
(W \an{u,u}_{\C^{2^{\lfloor \frac{d}{2} \rfloor}}}) u
\]
where $W$ is a self-adjoint operator that send the space of zonal functions into itself. The quantity $(W \an{u,u}_{\C^{2^{\lfloor \frac{d}{2} \rfloor}}}) $ remains zonal and thus the "zonality" is preserved by the Soler model. 

We recall that one needs to renormalize, which is a bit more involved in this context, we define the renormalisation in Section \ref{sec:Wick}. Note that the definition is slightly different from the one in \cite{deng2021invariant}, this is due to the fact that our interaction potential is not necessarily a Fourier multiplier. They would coincide in the case of a Fourier multiplier.

Finally, the Dirac operator is not bounded by below. This is an issue we did not manage to solve yet. We arbitrarily project onto the positive spectrum of the Dirac operator. We aim at handling this problem, for instance using Bogoliubov-Dirac-Fock (see \cite{chaix1989quantum, hainzl2005existence,borrelli2025global}) models in future works. Our equation is now
\begin{equation}\label{DiracFin}
i\partial_t u = iD_d u + P_+ [ :(W \an{u,u}_{\C^{2^{\lfloor \frac{d}{2} \rfloor}}}) u : ]
\end{equation}
where $P_+$ is the orthogonal projection onto the positive spectrum of $D_d$ and where $u$ is a zonal function.

\subsection{Main result}

As mentioned above, we follow the method depicted in \cite{oh2018pedestrian}, to construct a Gibbs measure for the Dirac zonal equation. To introduce the Gibbs measure, we set $\mu$ the probability measure on $H^{s}(\S^d)$, for $s<-\frac{1}{2}$ induced by the map 
\[
\omega\in\Omega \mapsto u(x)=u(x;\omega)=\sum_{n\in\N}\frac{g_n(\omega)}{\lambda_n}e_n,
\]
where $(g_n)_{n\in \N}$ is a sequence of independent standard complex-valued Gaussian random variables on a probability space $(\Omega,\mathcal{F},\mathbb{P})$, $\lambda_n:=\sqrt{\frac{d}{2}+n}$ for $n\in\N$ and $(e_n)_{n\in \N}$ is the family $(e^+_{d,n})_{n\in\N}$ defined in Section \ref{sec:defzonDirac}. 


In this paper, we construct a density measure with respect to $\mu$ denoted $\rho_\infty$ such that the following theorem holds

\begin{theorem}\label{th:mainth}
Let $s<-\frac{1}{2}$ and $q>12d$. We set $p\in [2,\infty)$ such that $\frac1p + \frac1q = \frac12$.
    
Let $W$ be a non negative integral operator whose kernel is non negative, which is continuous from $L^{p/2}(\S^d,\C)$ into $L^q(\S^d,\C)$ and which preserves zonality.


There exists a non trivial measure $\rho_\infty$ absolutely continuous with regard to $\mu$ and a random field $X$ with values in $\C(\R,H^s(\S^d))$ quch that $X$ is a solution in the distributional sense to \eqref{DiracFin} and such that for all $t\in \R$, the law of $X(t)$ is $\rho_\infty$.
\end{theorem}
    
\begin{remark}
    We only show the existence of a Gibbs measure for equation \eqref{DiracFin} but we do not show invariance of the measure, as we do not have uniqueness of solutions of the equation, thus we do not have the existence of a flow for the equation. The only thing that we can say, is that $\rho_\infty$ is a limit of invariant Gibbs measures associated to the truncated equations associated to \eqref{DiracFin} (one can see Section \ref{sec:Wick} for a definition of the truncated equations).
    
Unlike the wave equation, the semilinear Dirac equation does not present some natural regularisation of the nonlinearity. Unlike the Schrödinger equation, we also do not have bilinear Strichartz estimates which allow to get an improvement in the regularity of the first Picard iterate of the solution, the closest to such property may be found in \cite{candy2019multi}. Local smoothing properties such as \cite{boussaid2011virial} is not really smoothing. There seems to be no simple mechanism that would trigger the almost sure well-posedness of the equation. Perhaps one would have to rely on a quasi-linear argument such as \cite{burq2024probabilistic}.
\end{remark}

\begin{remark} The Bogoliubov--Dirac--Fock \cite{hainzl2005existence,borrelli2025global} model provides a nonnegative energy but it is restricted to density matrices that are orthogonal projections and approximations of the Dirac sea. Although there exists some randomization of density matrices in the literature, \cite{hadama2023probabilistic}, they are not convincing to build Gibbs measures.
\end{remark} 

\begin{remark}
Let $w$ be the integral kernel of $W$. If for all $x,y \in \S^2$ and for all $R \in \mathrm{SO}_n$ that leaves invariant the poles, we have that
\[
w(Rx,Ry) = w(x,y),
\]
then $W$ preserves zonality.

Since it is non negative, it is symmetric and we may chose $w(x,y) = w(y,x)$.

If $w \in L^{q}_x,L^{q/2}_y$ then $W$ maps continuously $L^{p/2}$ to $L^{q}$. 
\end{remark}



\section{Definition of the zonal Dirac operator}\label{sec:defzonDirac}

In this section, we give the minimal ingredients from \cite{camporesi1996eigenfunctions} to understand the construction of the Dirac operator on the sphere and its diagonalisation. 

More precisely, the sphere $\S^d$ has a warped product structure. It may be seen as the warped product between $[0,\pi]$ and the sphere $\S^{d-1}$. Because of this, the Dirac operator decomposes as something acting on $[0,\pi]$ and its version on the sphere $\S^{d-1}$. It is well known (see \cite{thaller2013dirac}) that the Dirac operator on the sphere admits a partial wave decomposition, that is to say that we may diagonalise it thanks to spherical harmonics. The function that acts on $\S^{d-1}$ as in the lowest spherical harmonics are called \emph{zonal}, the action of the Dirac operator on zonal function is called the \emph{zonal Dirac operator}. We now make all these statements rigorous.

\begin{definition}[Convention for $\Gamma$ matrices]
We define $(\Gamma_d^j)_{1\leq j\leq d}$ as a family of $\mathcal{M}_{2^{\lfloor \frac{d}{2}\rfloor}}(\C)$ by induction in the following way. First, we set
\[
\Gamma_1^1=1.
\]
Then, if $d>1$ is even, we set 
\[
\Gamma_d^d=\begin{pmatrix}
        0 & 1 \\ 1 & 0
    \end{pmatrix},
\]
with blocks of size $2^{\lfloor \frac{d-1}{2}\rfloor}$; for $j<d$, we set 
\[
\Gamma_d^j=\begin{pmatrix}
        0 & i\Gamma_{d-1}^j \\ -i\Gamma_{d-1}^j & 0
    \end{pmatrix}.
\]

If $d>1$ is odd, we set 
\[
\Gamma_d^d=\begin{pmatrix}
        1 & 0 \\ 0 & -1
    \end{pmatrix}
\]
with blocks of size $2^{\lfloor \frac{d-1}{2}\rfloor-1}$; for $j<d$, we set $\Gamma_d^j=\Gamma_{d-1}^j$.
\end{definition}

\begin{proposition}
    We have $\{ \Gamma_d^i,\Gamma_d^j\}:=\Gamma_d^i\Gamma_d^j+\Gamma_d^j\Gamma_d^i=2\delta^{ij},$ and $(\Gamma_d^i)^*=\Gamma_d^i$.
\end{proposition}

\begin{proof}
    The proof follows from straightforward computation and may be found in Section 2 of \cite{camporesi1996eigenfunctions}.
\end{proof}

Here, we do not define the Dirac operator on curved backgrounds, but the general construction may be found in \cite{parker2009quantum} and a general strategy to see its decomposition in warped product with a compact manifold may be found in \cite{ben2022global} and references therein.

In the following proposition, we use the structure of warped product of the sphere to decompose it nicely. We follow the conventions of \cite{camporesi1996eigenfunctions}.

\begin{proposition}[Dirac operator on $\S^d$]
For $d=1$ the Dirac operator $D_1$ acts on the set $F_1 = \{ f\in \mathcal C^1(\R, \C) | \forall \phi \in\R,f(\phi+2\pi)=-f(\phi)\}$ as $D_1 = \partial_\phi$.

For $d>1$, we represent an element $\Omega_d\in \S^d$ as $\Omega_d=(\cos\theta_d,\sin\theta_d \Omega_{d-1})$, with $\theta_d\in [0,\pi],\ \Omega_{d-1}\in \S^{d-1}$. When $d$ is even, the Dirac operator $D_d$ acts on $F_d = \mathcal C^1((0,\pi),\C^2) \otimes F_{d-1}$ as 
\[
D_d=(\partial_{\theta_d}+\frac{d-1}{2}\cot\theta_d)\Gamma_d^d+\frac{1}{\sin\theta_d}\begin{pmatrix}
    0 & iD_{d-1} \\ -iD_{d-1} & 0
\end{pmatrix}.
\]

For $d>1$ odd, the Dirac operator $D_d$ acts on $F_d = \mathcal C^1([0,\pi],\C) \otimes F_{d-1}$ as 
\[
D_d=(\partial_{\theta_d}+\frac{d-1}{2}\cot\theta_d)\Gamma^d_d+\frac{1}{\sin\theta_d}D_{d-1}.
\]
\end{proposition}

\begin{remark} Dirac operators are the only self-adjoint extensions of the operators written in partial derivatives as above. We write their domains $H^1(\S^d)$ but we want to emphasize that spinors are not invariant under rotations of angle $2\pi$, the action of a rotation of angle $2\pi$ on a spinor results in the change of sign of the spinor.

We want to precise our notation. If $d$ is even, a pure tensor in $F_d$ writes 
\[
\theta_d \mapsto \begin{pmatrix} f(\theta_d)\psi \\ g(\theta_d) \psi \end{pmatrix}
\]
where $f,g$ belong to $\mathcal C^1((0,\pi), \C)$ and $\psi \in F_{d-1}$. On these pure tensors, $D_d$ acts as 
\[
D_d\begin{pmatrix} f\psi \\ g\psi \end{pmatrix} = \begin{pmatrix} 
(\partial_{\theta_d}+\frac{d-1}{2}\cot\theta_d) g \psi + i \frac{g}{\sin \theta_d} D_{d-1} \psi \\
(\partial_{\theta_d}+\frac{d-1}{2}\cot\theta_d) f \psi - i \frac{f}{\sin \theta_d} D_{d-1} \psi
\end{pmatrix} . 
\]
\end{remark}

\begin{definition}[Zonal functions]
By induction, we define $\psi_d^\pm$ as \begin{itemize}
\item $\psi_1^\pm : \phi\mapsto e^{\pm i\phi/2}$; 
\item for $d>1$ even 
\[
\psi_d^\pm=\begin{pmatrix}
        \cos\frac{\theta_d}{2}\psi_{d-1}^-\\ \pm i \sin \frac{\theta_d}{2}\psi_{d-1}^-
    \end{pmatrix};
\]
\item for $d>1$ odd 
\[
\psi_d^\pm =\cos\frac{\theta_d}{2}(1+i\Gamma_d^d)\psi_{d-1}^-\pm i\sin \frac{\theta_d}{2}(1+i\Gamma_d^d)\psi_{d-1}^+.
\]
\end{itemize}
We remark that $\psi_d^\pm \in F_d$. 

For $d>1$ even (resp. odd) say that a function $\varphi \in \mathcal C((0,\pi),\C^2) \otimes F_{d-1}$ (resp. $\mathcal C((0,\pi),\C) \otimes F_{d-1}$) is zonal if it writes 
\[
\varphi : \theta_d \mapsto 
\begin{pmatrix} \varphi_+(\theta_d) \psi_{d-1}^- \\
i \varphi_-(\theta_d) \psi_{d-1}^-
\end{pmatrix} 
\]
\bigg(resp. 
\[
\varphi_+(\theta_d) (1+ i\Gamma_d^d) \psi_{d-1}^- +i\varphi_-(\theta_d) (1+ i\Gamma_d^d) \psi_{d-1}^+ .\bigg)
\]
In both cases, we set $\varphi = [\varphi_+, \varphi_-]$. We write the set of zonal functions $Z_d$.
\end{definition}

\begin{proposition}\label{prop:actionDiraczonal}
The image of $Z_d\cap F_d$  under the action of the Dirac operator is included in $Z_d$. What is more, the Dirac operator acts on zonal functions in the following way : 
\[
D_d([\varphi_+,\varphi_-]) = \Big[ i\Big(\partial_{\theta_d} + \frac{d-1}{2} \cot(\theta_d)+ \frac{d-1}2 \frac1{\sin \theta_d} \Big) \varphi_-, i\Big(-\partial_{\theta_d} - \frac{d-1}{2} \cot(\theta_d)+ \frac{d-1}2 \frac1{\sin \theta_d} \Big) \varphi_+ \Big].
\]
Finally, for all $\varphi \in Z_d$ and $d$ even, we have that $\an{\varphi,\varphi}_{\C^{2^{\lfloor \frac{d}{2} \rfloor}}} $ is proportional to $|\varphi_+|^2 + |\varphi_-|^2$ and $\an{\varphi,\Gamma_d^d \varphi}_{\C^{2^{\lfloor \frac{d}{2} \rfloor}}}$ belongs to $\mathcal C((0,\pi),\C)$ (it only depends on $\theta_d$).
\end{proposition}

\begin{lemma}\label{lem:actionDiracZonal} We have for all $d\in \N^*$,
    \begin{enumerate}
        \item $D_d\psi_d^\pm=\pm i \frac{d}{2}\psi_d^\pm$;
        \item $\langle \psi_d^\pm,\psi_d^\pm\rangle_{\C^{2^{\lfloor \frac{d}{2}\rfloor}}}$ is a constant; 
        \item if $d$ is even, $\Gamma_{d+1}^{d+1}\psi_{d}^\pm=\psi_d^\mp$.
    \end{enumerate}
\end{lemma}

\begin{remark}
    Note that this proposition and this lemma are proven in \cite{camporesi1996eigenfunctions} but not only for zonal functions. Therefore, the proof we summarise here in the context of zonal functions is somewhat simplified. We include it for seek of completeness.
\end{remark}

\begin{proof}[Proof of Proposition \ref{prop:actionDiraczonal} and Lemma \ref{lem:actionDiracZonal}] We prove both the lemma and the proposition at the same time by induction on $d$. For $d=1$, the proposition does not hold but we prove the lemma. We have
\[
D_1 \psi_1^\pm = \pm \frac{i}{2} \psi_1^\pm, \quad |\psi_1^\pm|^2 = 1
\]

Assume that the lemma is valid in dimension $d-1$ we prove the proposition and the lemma for dimension $d>1$. 

Case 1 : $d$ is even. We deduce that
\[
D_d[\varphi_+,\varphi_-] = \begin{pmatrix} 
(\partial_{\theta_d}+\frac{d-1}{2}\cot\theta_d) i\varphi_- \psi^-_{d-1} + i \frac{i \varphi_-}{\sin \theta_d} D_{d-1} \psi^-_{d-1} \\
(\partial_{\theta_d}+\frac{d-1}{2}\cot\theta_d) \varphi_+ \psi^-_{d-1} - i \frac{\varphi_+}{\sin \theta_d} D_{d-1} \psi^-_{d-1}
\end{pmatrix} .
\]
Since $D_{d-1} \psi^-_{d-1} = -i \frac{d-1}{2} \psi^-_{d-1}$ we deduce
\[
D_d[\varphi_+,\varphi_-] = \begin{pmatrix} 
i \Big((\partial_{\theta_d}+\frac{d-1}{2}\cot\theta_d) \varphi_- + \frac{d-1}{2} \frac{ \varphi_-}{\sin \theta_d} \Big) \psi^-_{d-1} \\
i^2\Big( -(\partial_{\theta_d}+\frac{d-1}{2}\cot\theta_d) \varphi_+ + \frac{d-1}{2} \frac{\varphi_+}{\sin \theta_d} \Big) \psi^-_{d-1}
\end{pmatrix} 
\]
which corresponds to the desired action of $D_d$. Since $[\varphi_+,\varphi_-] = (\varphi_+,i\varphi_-) \otimes \psi_{d-1}^-$, we deduce 
\[
\an{\varphi,\varphi}_{\C^{2^{\lfloor \frac{d}{2} \rfloor}}} = \an{(\varphi_+,i\varphi_-),(\varphi_+,i\varphi_-)}_{\C^2} \an{\psi_{d-1}^-,\psi_{d-1}^-}_{\C^{2^{\lfloor \frac{d-1}{2} \rfloor}}}  = c_{d-1} (|\varphi_+|^2 + |\varphi_-|^2).
\]
A simple computation gives that $$\Gamma_d^d [\varphi_+,\varphi_-]=\begin{pmatrix}
        1 & 0 \\ 0 & -1
    \end{pmatrix}\begin{pmatrix} \varphi_+(\theta_d) \psi_{d-1}^- \\
i \varphi_-(\theta_d) \psi_{d-1}^-
\end{pmatrix} =\begin{pmatrix} i\varphi_-(\theta_d) \psi_{d-1}^- \\
 \varphi_+(\theta_d) \psi_{d-1}^-
\end{pmatrix}  = [i\varphi_-,-i\varphi_+ ].$$ From this result we deduce that $\an{\varphi,\Gamma_d^d\varphi}_{\C^{2^{\lfloor \frac{d}{2} \rfloor}}}$ is proportional to $2 \mathrm{Im} (\overline{\varphi_-} \varphi_+)$, so it only depends on $\theta_d$.

To get the lemma, we apply the proposition with $\varphi_+ = \cos(\cdot/2)$ and $\varphi_- = \pm \sin(\cdot /2)$. A straightforward computation yields
\begin{align*}
\Big( \partial_\theta + \frac{d-1}{2} (\cot \theta + \sin^{-1} \theta) \Big) \sin \frac\theta2 =& \frac{d}{2}\cos\frac\theta2\\
\Big( -\partial_\theta + \frac{d-1}{2} (-\cot \theta + \sin^{-1} \theta) \Big) \cos \frac\theta2 =& \frac{d}{2}\sin\frac\theta2   
\end{align*}
which yields the result. 

We have 
\[
\langle \psi_d^\pm,\psi_d^\pm\rangle_{\C^{2^{\lfloor \frac{d}{2}\rfloor}}} = c_{d-1} (\cos^2(\frac\theta2) + \sin^2(\frac\theta2) ) = c_{d-1}.
\]

Finally $\Gamma_{d+1}^{d+1} [\varphi_+,\varphi_-] = [\varphi_+,-\varphi_-]$, we get $\Gamma_{d+1}^{d+1}\psi_d^\pm = \psi_d^\mp$.

Case 2 : $d$ is odd. We have that
\begin{multline*}
    D_d[\varphi_+,\varphi_-] = \Big( (\partial_{\theta_d} + \frac{d-1}{2} \cot \theta_d) \Gamma_d^d  + \frac1{\sin \theta_d} D_{d-1} \Big) \varphi_+ (1+i\Gamma_d^d)\psi_{d-1}^- \\
    +i \Big( (\partial_{\theta_d} + \frac{d-1}{2} \cot \theta_d) \Gamma_d^d  + \frac1{\sin \theta_d} D_{d-1} \Big) \varphi_- (1+i\Gamma_d^d)\psi_{d-1}^+.
\end{multline*}
We remark that $\Gamma_d^d$ and $1+\Gamma_d^d$ commute, which yields $\Gamma_d^d(1+i\Gamma_d^d) \psi^\pm_{d-1} = (1+i\Gamma_d^d)\psi^\mp_{d-1}$. What is more, 
\[
D_{d-1}(1+i\Gamma_d^d) \psi^\pm_{d-1} = D_d(\psi^\pm_{d-1} +i \psi^\mp_{d-1})= i\frac{d-1}{2} (\pm \psi^\pm_{d-1} \mp i \psi^\mp_{d-1}) = \pm \frac{d-1}{2} (i\Gamma_d^d +1) \psi^\mp_{d-1}.
\]
This yields
\begin{multline*}
    D_d[\varphi_+,\varphi_-] = \Big( (\partial_{\theta_d} + \frac{d-1}{2} \cot \theta_d)  - \frac{d-1}{2\sin \theta_d}  \Big) \varphi_+ (1+i\Gamma_d^d)\psi_{d-1}^+ \\
    +i \Big( (\partial_{\theta_d} + \frac{d-1}{2} \cot \theta_d)  + \frac{d-1}{2\sin \theta_d} \Big) \varphi_- (1+i\Gamma_d^d)\psi_{d-1}^-
\end{multline*}
which we can rearrange as
\begin{multline*}
    D_d[\varphi_+,\varphi_-] = i \Big( (\partial_{\theta_d} + \frac{d-1}{2} \cot \theta_d)  + \frac{d-1}{2\sin \theta_d} \Big) \varphi_- (1+i\Gamma_d^d)\psi_{d-1}^- \\
    +i^2\Big( -(\partial_{\theta_d} + \frac{d-1}{2} \cot \theta_d)  + \frac{d-1}{2\sin \theta_d}  \Big) \varphi_+ (1+i\Gamma_d^d)\psi_{d-1}^+
\end{multline*}
which we recognise to be 
\begin{multline*}
    D_d[\varphi_+,\varphi_-] = \Big[ i \Big( (\partial_{\theta_d} + \frac{d-1}{2} \cot \theta_d)  + \frac{d-1}{2\sin \theta_d} D_{d-1} \Big) \varphi_- ,\\
    i\Big( -(\partial_{\theta_d} + \frac{d-1}{2} \cot \theta_d)  + \frac{d-1}{2\sin \theta_d}  \Big) \varphi_+ \Big].
\end{multline*}

Using that 
\begin{align*}
    [\varphi_+,\varphi_-] = &\varphi_+ \otimes (1+i\Gamma_d^d) \psi_{d-1}^- + i\varphi_- \otimes (1+i\Gamma_d^d) \psi_{d-1}^+\\
    \an{(1+i\Gamma_d^d)\psi_{d-1}^+,(1+i\Gamma_d^d)\psi_{d-1}^+}_{\C^{2^{\lfloor \frac{d}{2} \rfloor}}}= &\an{(1+i\Gamma_d^d)\psi_{d-1}^-,(1+i\Gamma_d^d)\psi_{d-1}^-}_{\C^{2^{\lfloor \frac{d}{2} \rfloor}}}=2c_{d-1}\\ \an{(1+i\Gamma_d^d)\psi_{d-1}^+,(1+i\Gamma_d^d)\psi_{d-1}^-}_{\C^{2^{\lfloor \frac{d}{2} \rfloor}}} = & 2c_{d-2}\cos \theta_{d-1}
\end{align*}
we get
\[
\an{\varphi,\varphi}_{\C^{2^{\lfloor \frac{d}{2} \rfloor}}} = 2c_{d-1}(|\varphi_+|^2 + |\varphi_-|^2)
+ 4c_{d-2}\cos \theta_{d-1} \mathrm{Im} (\overline{\varphi_-} \varphi_+). 
\]
Applying these formulae to $\varphi_+ = \cos\frac{\cdot}{2}$ and $\varphi_- = \sin (\frac{\cdot}{2})$, we get that
\[
D_d \psi_d^\pm = \pm i\frac{d}{2} \psi_d^\pm,\quad \an{\psi_d^\pm,\psi_d^\pm} = 2 c_{d-1}.
\]

\end{proof}

\begin{definition}
Let $\an{}_{Z_d}$ be the scalar product defined on $Z_d$ by
\[
\an{f,g}_{Z_d} = c_{d-1}\int_{0}^\pi (\bar f_+(\theta) g_+(\theta) + \bar f_-(\theta) g_-(\theta)) \sin^{d-1}\theta d\theta
\]
and set $L^2_\zon$ the completion of $Z_d$ for the topology induced by this scalar product.
\end{definition}

\begin{remark} This is consistent with the fact that depending on the parity of $d$, we have for $f=[f_+,f_-]$ and $g=[g_+,g_-]$, either
\[
\an{f,g}_{\C^{2^{\lfloor \frac{d}{2}\rfloor}}} = c_{d-1}(\bar f_+ (\theta_d) g_+(\theta_d) +  \bar f_-( \theta_d) g_-(\theta_d))
\]
or
\[
\an{f,g}_{\C^{2^{\lfloor \frac{d}{2}\rfloor}}} =2 c_{d-1}(\bar f_+ (\theta_d) g_+(\theta_d) +  \bar f_-( \theta_d) g_-(\theta_d)) + 2 i c_{d-2} \cos(\theta_{d-1}) (\bar f_+ (\theta_d) g_-(\theta_d)-\bar f_-(\theta_d)g_+(\theta_d)) 
\]
and 
\[
\int_0^\pi \cos(\theta_{d-1})\sin^{d-2}(\theta_{d-1}) d\theta_{d-1} = 0.
\]
\end{remark}

\begin{proposition} For $n\in \N$, set 
\[
e_{d,n}^\pm : \theta \mapsto  c_{d,n} [ P_n^{(\frac{d}{2}-1, \frac{d}{2})}(\cos \theta) \cos(\frac\theta2), \mp P_n^{(\frac{d}{2}-1, \frac{d}{2})}(\cos \theta) \sin(\frac\theta2)] 
\]
where $P_n^{(\alpha, \beta)}$ are Jacobi polynomials and $c_{d,n}$ is a constant chosen such that the $e_{d,n}^\pm$ are normalised in $L^2_\zon$.

We have that $e_{d,n}^\pm \in Z_d\cap F_d$, that $\an{e_{d,n_1}^{\iota_1}, e_{d,n_2}^{\iota_2}}_{Z_d} = \delta_{n_1}^{n_2} \delta_{\iota_1}^{\iota_2}$, that 
\[
D_d e_{d,n}^\pm = \mp i \Big( n + \frac{d}{2} \Big) e_{d,n}^\pm
\]
and finally that $(e_{d,n}^\pm)_n$ spans $L^2_\zon$.
\end{proposition}

\begin{remark} We have that $e_{d,0}^\pm = \psi_d^\pm$.
\end{remark}

\begin{proof} We do not give a complete proof, which may be found in \cite{camporesi1996eigenfunctions} but we give an idea why the Jacobi polynomials are involved.

We wish to diagonalise $D_d$. For that, let $\varphi \in F_d\cap Z_d$ be such that $D_d\varphi=i\lambda \varphi$. This is equivalent to 
\[
\left \lbrace{ \begin{array}{c}
i(\partial_\theta +\frac{d-1}{2}\cot\theta+\frac{d-1}{2}\frac{1}{\sin\theta})\varphi_-=i\lambda\varphi_+\\
i(\partial_\theta +\frac{d-1}{2}\cot\theta-\frac{d-1}{2}\frac{1}{\sin\theta})\varphi_+=-i\lambda\varphi_-
\end{array} } \right. .
\]

We set $\varphi_-=\sin\frac{\theta}{2}f_-(\cos\theta)$ and $\varphi_+=\cos\frac{\theta}{2}f_+(\cos\theta)$ 

We obtain 
\[
\left \lbrace{ 
\begin{array}{c}
     -2\sin^2\frac{\theta}{2}f_-'(\cos \theta)+\frac{d}{2}f_-(\cos \theta)=\lambda f_+ (\cos \theta) \\
     -2\cos^2\frac{\theta}{2}f_+'(\cos \theta)-\frac{d}{2}f_+(\cos \theta)=-\lambda f_-(\cos \theta)
\end{array} }
\right. .
\]

Let $z=\cos\theta$, we deduce the system
\[
\left \lbrace{ \begin{array}{c}
     (z-1)f_-'+\frac{d}{2}f_-=\lambda f_+  \\
     (z+1)f_+'+\frac{d}{2}f_+=\lambda f_-
\end{array}}
\right. .
\]

By combining the two equations of order 1, we end up with one equation of order 2 :
\[
(1-z^2)f_+''-(-1+(d+1)z)f_+'+(\lambda^2-\frac{d^2}{4})f_+= 0
\]
or alternatively
\[
(1-z^2)f_-''-(1+(d+1)z)f_-'+(\lambda^2-\frac{d^2}{4})f_-= 0
\]
One can solve this equation using Jacobi's polynomials $f_+=P_n^{(\alpha,\beta)}$, with $\beta=\frac{d}{2},\ \alpha=\frac{d}{2}-1$ and $n$ such that 
\[
n(n+d)=\lambda^2-\frac{d^2}{4},
\]
that is $\lambda=\pm (n+\frac{d}{2})$. We deduce $f_-$ from the equation 
\[
\lambda f_- = (z+1)f_+'+ \frac{d}{2} f_+
\]

We get $f_+=P_n^{(\frac{d}{2}-1,\frac{d}{2})}$ and $f_-=P_n^{(\frac{d}{2},\frac{d}{2}-1)}sgn(\lambda)$.

\end{proof}

\section{Wick renormalization and Gibbs measure}\label{sec:Wick}

\begin{notation}Let $d$ be an even number. With the notations of Section \ref{sec:defzonDirac}, we set for $n\in \N$, $e_n = e_{d,n}^+$ (we drop the dependence on the dimension and on the positive spectrum). Let $E_N$ be the subspace of $L^2_\zon$ generated by $(e_n)_{n\leq N}$; let $P_N$ be the orthogonal projection on $E_N$. Let $\lambda_n = \sqrt{\frac{d}{2} + n}$.

Let $(g_n)_{n\in \N}$ be a sequence of iid complex Gaussian variables, centred and normalised. We set 
\[
\varphi_N=\sum_{n\leq N}g_n e_n \frac{1}{\lambda_n}.
\]
Let $\sigma_N(x,y)=\E[\lvert \varphi_N(x)\rangle\langle\varphi_N(y)\rvert]$ (the ketbra notation is with regard to $\C^{2^{\lfloor \frac{d}{2}\rfloor}}$), and $\sigma_N(x)=\E[\lvert\varphi_N\rvert^2(x)]$. The sequence $\varphi_N$ converges in $L^2(\Omega, H^s(\S^d))$ for any $s<0$. We write its limit $\varphi$ and the law of $\varphi$ on the inductive limit $\cup_{s<0} H^s$ is denoted by $\mu$.

From now on, if $u \in \C^{2^{\lfloor \frac{d}{2} \rfloor}}$, we use $|u|^2$ for $\an{u,u}_{\C^{2^{\lfloor \frac{d}{2} \rfloor}}}$ and $\bar u$ for the linear form $\langle u|_{\C^{2^{\lfloor \frac{d}{2} \rfloor}}}$.
\end{notation}

We approach equation \eqref{DiracFin} by 
\begin{equation}\label{DiracN}
i\partial_t u = i D_d u + F_N(u) 
\end{equation}
where
\begin{equation}\label{def:NthorderNL}
    F_N(u)=P_N(:W\lvert P_N u \rvert^2 P_N u :),
\end{equation}
with
\[
:W(\lvert u \rvert^2) u : (x) = \int dy w(x,y)\Big[ |u|^2(y) u(x) - \sigma_N(y) u(x) - \sigma_N(x,y) u(y)\Big] 
\]

Equation \eqref{DiracN} can be decomposed into a linear equation $i\partial_t u = i D_d u $ on $E_N^\bot$ and a Hamiltonian ODE on $E_N$. The Hamiltonial on $E_N$ is given by
\[
H_N(u) = \frac12 \an{u,iD_d u}_{\S^d} + \frac14 \mathcal E_N(u)
\]
with
\begin{equation}\label{def:energie}\begin{array}{rcl}
     \mathcal{E}_N(\psi)& = & \int_{\S^d} dxdy w(x,y) \Big[ \lvert \psi (x)\rvert^2\lvert \psi(y)\rvert^2-2\lvert \psi (x)\rvert^2\sigma_N(y) \\
     & & - 2\overline{\psi (x)}\sigma_N(x,y)\psi+\sigma_N(x)\sigma_N(y)+Tr(\sigma_N(x,y)\sigma_N(y,x))\Big].
\end{array}
\end{equation}

The rest of this section is devoted to proving that $(\mathcal E_N \circ P_N)_N$, $(e^{-\mathcal E_N\circ P_N})_N$ and $(F_N)_N$ are Cauchy sequences in $L^r_\mu$ for $r\in [1,\infty)$.

\subsection{Renormalisation of the energy}


\begin{proposition}\label{prop;G_N Cauchy}
The sequence $(\mathcal{E}_N\circ P_N)_{N\in\N}$ is a Cauchy sequence in $L^2(\mu)$.

Moreover, for all $0<\nu <1 - \frac{6d}q$, we have that for any $M\leq N $
    
    \begin{equation}\label{eq:cauchybound}
       \|\mathcal{E}_N\circ P_N - \mathcal{E}_M\circ P_M \|_{L^2_\mu}\lesssim M^{-\nu/2}.
    \end{equation}
\end{proposition}

We first set
\begin{equation}\begin{array}{rcl}
G_N & = & \int dxdy  w(x,y) \Big[\lvert \varphi_N (x)\rvert^2 \lvert \varphi_N(y)\rvert^2-2\sigma_N(y)\lvert \varphi_N (x) \rvert^2-2\overline{\varphi_N (x)}\sigma_N(x,y)\varphi_N  (y) \\
     & &  +\sigma_N(x)\sigma_N(y)+Tr(\sigma_N(x,y)\sigma_N(y,x))\Big],
\end{array}
\end{equation}
We remark that for any $M,N \in \N$ and any non negative measurable $f : \C^2 \mapsto \R_+$, we have 
\[
\|f\circ(\mathcal{E}_N\circ P_N, \mathcal E_M \circ P_M)\|_{L^2(\mu)}^2=\E[\lvert f(G_N,G_M)\rvert^2].
\]


The proof of Proposition \ref{prop;G_N Cauchy} is based on the following identity.
\begin{lemma}
    We have the following equality 
\begin{equation}\label{eq:Gn}
        G_N=  \sum_{j,k,l,m\leq N} \int dxdy w(x,y)\frac{\langle e_j,e_k\rangle (x)\langle e_l,e_m\rangle (y)}{\lambda_j \lambda_k\lambda_l\lambda_m}:\overline{g_j}g_k\overline{g_l}g_m:,
\end{equation}
where $:\overline{g_j}g_k\overline{g_l}g_m: =\Bar{g_j}g_k\overline{g_l}g_m-\delta_l^m\overline{g_j}g_k-\delta_j^k\overline{g_l}g_m-\delta_j^mg_k\overline{g_l}-\delta_k^l \overline{g_j}g_m+\delta_j^k \delta_l^m+\delta_l^k\delta_j^m$.
\end{lemma}

\begin{proof}
Set 
\[
G_N = A- B-C+D+E
\]
with
\begin{align*}
 A = & \int dxdy  w(x,y) \lvert \varphi_N (x)\rvert^2 \lvert \varphi_N(y)\rvert^2\\
 B = & 2 \int dxdy  w(x,y) \sigma_N(y)\lvert \varphi_N (x) \rvert^2 \\
 C = & 2\int dxdy  w(x,y) \overline{\varphi_N (x)}\sigma_N(x,y)\varphi_N  (y) \\
 D = & \int dxdy  w(x,y) \sigma_N(x)\sigma_N(y)\\
 E = & \int dxdy  w(x,y) Tr(\sigma_N(x,y)\sigma_N(y,x)).
\end{align*}
    
First, we have by definition of $\varphi_N$,
\[
A = \sum_{j,k,l,m \leq N } \int dxdy w(x,y) \frac{\langle e_j,e_k\rangle (x) \langle e_l,e_m\rangle (y)}{\lambda_j\lambda_k\lambda_l\lambda_m}\overline{g_j}g_k\overline{g_l}g_mdx.
\]

For the second term, because $w$ is symmetric we have,
\[
B = \int dxdy  w(x,y) [\lvert \varphi_N (x) \rvert^2\sigma_N(y) + \sigma_N(x)|\varphi_N(y)|^2.
\]
We replace $\varphi_N$ and $\sigma_N$ by their respective values and get
\[
B = \int dxdy  w(x,y) \Big[\sum_{j,k\leq N} \frac{\an{e_j,e_k}(x)}{\lambda_j \lambda_k}\overline{g_j}g_k \sum_{l\leq N} \frac{|e_l(y)|^2}{\lambda_l^2} + \sum_{j\leq N}\frac{|e_j(x)|^2}{\lambda_j^2}\sum_{l,m \leq N}\frac{\an{e_l,e_m}(y) \overline{g_l} g_m}{\lambda_l\lambda_m}\Big].
\]
Using the Kronecker symbol, we get
\[
B = \int dxdy  w(x,y) \Big[\sum_{j,k\leq N} \frac{\an{e_j,e_k}(x)}{\lambda_j \lambda_k}\overline{g_j}g_k \sum_{l,m\leq N}\delta_l^m \frac{\an{e_l,e_m}(y)}{\lambda_l \lambda_m} + \sum_{j,k\leq N}\delta_k^j\frac{\an{e_j,e_k}(x)}{\lambda_j \lambda_k}\sum_{l,m \leq N}\frac{\an{e_l,e_m}(y) \overline{g_l} g_m}{\lambda_l\lambda_m}\Big]
\]
which we rearrange as
\[
B = \sum_{j,k,l,m\leq N} \int dxdy  w(x,y)  \frac{\an{e_j,e_k}(x)\an{e_l,e_l}}{\lambda_j \lambda_k\lambda_l\lambda_m}[\overline{g_j}g_k \delta_l^m + \delta_j^k \overline{g_l}g_m] .
\]

We turn to $C$. Again, we may use the symmetry of $w$ to get
\[
 C =  \int dxdy  w(x,y) [\overline{\varphi_N (x)}\sigma_N(x,y)\varphi_N  (y) + \overline{\varphi_N (y)} \sigma_N (y,x) \varphi_N(x) ].
\]
Using the definitions of $\sigma_N$ and $\varphi_N$, we get
\begin{multline*}
 C =  \int dxdy  w(x,y) \Big[ \sum_{j\leq N} \frac{\overline{g_j}\langle e_j(x)| }{\lambda_j} \sum_{k\leq N} \frac{|e_k(x) \rangle \langle e_k(y) |}{\lambda_k^2} \sum_{m\leq N} \frac{g_m e_m(y)}{\lambda_m}
 \\
 +\sum_{l\leq N} \frac{\overline{g_l}\langle e_l(y)| }{\lambda_l} \sum_{j\leq N} \frac{|e_j(y) \rangle \langle e_j(x) |}{\lambda_j^2} \sum_{k\leq N} \frac{g_k e_k(x)}{\lambda_k}\Big].
\end{multline*}
Making use of the Kronecker symbols, we get
\begin{multline*}
 C =  \int dxdy  w(x,y) \Big[ \sum_{j\leq N} \frac{\overline{g_j}\langle e_j(x)| }{\lambda_j} \sum_{k,l\leq N}\delta_k^l \frac{|e_k(x) \rangle \langle e_l(y) |}{\lambda_k\lambda_l} \sum_{m\leq N} \frac{g_m e_m(y)}{\lambda_m}\\
 +\sum_{l\leq N} \frac{\overline{g_l}\langle e_l(y)| }{\lambda_l} \sum_{j,m\leq N}\delta_j^m \frac{|e_m(y) \rangle \langle e_j(x) |}{\lambda_j \lambda_m} \sum_{k\leq N} \frac{g_k e_k(x)}{\lambda_k}\Big]
\end{multline*}
which we recognize to be
\[
 C =  \sum_{j,k,lm\leq N}\int dxdy  w(x,y) \frac{\an{e_j,e_k}(x) \an{e_l,e_m}(y)}{\lambda_j\lambda_k\lambda_l \lambda_m} [\overline{g_j} \delta_k^l g_m + \delta_{j}^m \overline{g_l}g_k]. 
\]

We use the definition of $\sigma_N$ and Kronecker symbols to obtain
\begin{align*}
 D = & \int dxdy  w(x,y) \sigma_N(x)\sigma_N(y) \\
 = & \sum_{j,l \leq N}\int dxdy  w(x,y) \frac{|e_j(x)|^2}{\lambda_j^2} \frac{|e_l(y)|^2}{\lambda_l^2} \\ 
 = & \sum_{j,k,l,m\leq N} \int dxdy  w(x,y) \delta_j^k \delta_l^m \frac{\an{e_j,e_k}(x) \an{e_l,e_m}(y)}{\lambda_j\lambda_k  \lambda_l \lambda_m}.
\end{align*}
We proceed in the same way for $E$ and get
\begin{align*}
 E = & \int dxdy  w(x,y) \mathrm{Tr}(\sigma_N(x,y)\sigma_N(y,x)) \\
 & \sum_{j,k \leq N}\int dxdy  w(x,y) \mathrm{Tr}\Big( \frac{|e_k(x)\rangle \langle e_k(y) |}{\lambda_k^2} \frac{|e_j(y)\rangle \langle e_j(x)| }{\lambda_j^2}\Big) \\ 
 & \sum_{j,k \leq N}\int dxdy  w(x,y)  \frac{\an{e_j,e_k}(x) \an{e_k,e_j}(y) }{\lambda_j^2 \lambda_k^2} \\ 
 & \sum_{j,k,l,m \leq N}\int dxdy  w(x,y) \delta_j^m \delta_k^l \frac{\an{e_j,e_k}(x) \an{e_l,e_m}(y) }{\lambda_j \lambda_k \lambda_l \lambda_m} .
\end{align*}

\end{proof}

\begin{lemma}
    For $M\leq N$, we have that \begin{equation}\label{eq:espgn2}
        \E[\lvert G_N - G_M \rvert^2]  \leq 4 \sum_{(j,k,l,m)\in I_{M,N}}\frac{1}{\lambda_j^2\lambda_k^2\lambda_l^2\lambda_m^2} |A_{j,k,l,m}|^2,
    \end{equation}
    where \begin{equation}
        A_{j,k,l,m}=\int dxdy w(x,y) \langle e_j,e_k\rangle (x) \langle e_l,e_m\rangle(y) 
    \end{equation}
and 
\[
I_{M,N} = \{ (j,k,l,m) \;|\; M < \max(j,k,l,m) \leq N \}.
\]
\end{lemma}

\begin{proof}
    The equality comes from by noticing that $:\overline{g_j}g_k\overline{g_l}g_m:$ corresponds to the definition of the Wick product in \cite{janson1997gaussian}. Using Theorem 3.9 p.26 in \cite{janson1997gaussian} gives
\[
\E[:\overline{g_j}g_k\overline{g_l}g_m::g_{j'}\overline{g_{k'}}g_{l'}\overline{g_{m'}}:]=(\delta_j^{j'}\delta_l^{l'}+\delta_j^{l'}\delta_l^{j'})(\delta_k^{k'}\delta_m^{m'}+\delta_k^{m'}\delta_m^{k'}).
\]
    
Using this expression combined with \eqref{eq:Gn} and the fact that $I_{M,N}$ is stable under permutations yield
    \begin{equation*}
         \E[\lvert G_N - M \rvert^2]  = \sum_{j,k,l,m\in I_{M,N}}\frac{1}{\lambda_j^2\lambda_k^2\lambda_l^2\lambda_m^2}\Bigg(\lvert A_{j,k,l,m}\rvert^2+A_{j,k,l,m}A_{m,j,k,l}+A_{j,k,l,m}A_{k,l,m,j}+A_{j,k,l,m}A_{m,l,k,j}\Bigg).
    \end{equation*}
    
We use that for any permutation $\sigma \in \mathfrak S_4$, we have 
\[
|A_{j,k,l,m}A_{\sigma(j), \sigma(k),\sigma(l), \sigma(m)}|\leq \frac{|A_{j,k,l,m}|^2 + |A_{\sigma(j), \sigma(k),\sigma(l), \sigma(m)}|^2}{2}
\]
and the symmetry of the sum to conclude.

\end{proof}

\begin{proof}[Proof of Proposition \ref{prop;G_N Cauchy}]
Let $\nu \in (0,1-\frac{6d}{q})$. Consider the sum
\[
\sum_{(j,k,l,m)\in I_{M,N}}\frac{1}{\lambda_j^2\lambda_k^2\lambda_l^2\lambda_m^2} |A_{j,k,l,m}|^2.
\]
Let $(j,k,l,m) \in I_{M,N}$ and let $i_0 = \max (j,k,l,m)$. We have that
\[
\lambda_{i_0}^{-2} \leq M^{-\nu} \prod_{i\neq i_0} \lambda_i^{-2(1-\nu)/3}.
\]
We deduce 
\[
\frac{1}{\lambda_j^2\lambda_k^2\lambda_l^2\lambda_m^2} \leq M^{-\nu }\prod_{i\neq i_0}\lambda_i^{-2(1+(1-\nu)/3)}
\] 
and therefore
\begin{align*}
\sum_{(j,k,l,m)\in I_{M,N}}\frac{1}{\lambda_j^2\lambda_k^2\lambda_l^2\lambda_m^2} |A_{j,k,l,m}|^2 \leq  & M^{-\nu} \sum_{j,k,l} \frac1{\lambda_j^{2(1+(1-\nu)/3)} \lambda_k^{2(1+(1-\nu)/3)} \lambda_l^{2(1+(1-\nu)/3)}}\sum_m |A_{j,k,l,m}|^2\\
& +  M^{-\nu} \sum_{j,k,m} \frac1{\lambda_j^{2(1+(1-\nu)/3)} \lambda_k^{2(1+(1-\nu)/3)} \lambda_m^{2(1+(1-\nu)/3)}} \sum_l |A_{j,k,l,m}|^2 \\
& +  M^{-\nu} \sum_{j,l,m} \frac1{\lambda_j^{2(1+(1-\nu)/3)} \lambda_l^{2(1+(1-\nu)/3)} \lambda_m^{2(1+(1-\nu)/3)}} \sum_k |A_{j,k,l,m}|^2 \\
& +  M^{-\nu} \sum_{k,l,m} \frac1{\lambda_k^{2(1+(1-\nu)/3)} \lambda_l^{2(1+(1-\nu)/3)} \lambda_m^{2(1+(1-\nu)/3)}} \sum_j |A_{j,k,l,m}|^2
.
\end{align*}

We have that
\begin{align*}
A_{j,k,l,m} = &\an{W(\an{e_k,e_j})e_l,e_m}_{L^2} \\
= & \an{e_l,W(\an{e_j,e_k}) e_m}_{L^2} \\
= & \an{W(\an{e_m,e_l}) e_j,e_k}_{L^2} \\
= & \an{e_j, W(\an{e_l,e_m})e_k}_{L^2}.
\end{align*}
We deduce 
\begin{align*}
\sum_{(j,k,l,m)\in I_{M,N}}\frac{1}{\lambda_j^2\lambda_k^2\lambda_l^2\lambda_m^2} |A_{j,k,l,m}|^2 \leq  & M^{-\nu} \sum_{j,k,l} \frac1{\lambda_j^{2(1+(1-\nu)/3)} \lambda_k^{2(1+(1-\nu)/3)} \lambda_l^{2(1+(1-\nu)/3)}}\|W(\an{e_k,e_j})e_l \|_{L^2}^2\\
& +  M^{-\nu} \sum_{j,k,m} \frac1{\lambda_j^{2(1+(1-\nu)/3)} \lambda_k^{2(1+(1-\nu)/3)} \lambda_m^{2(1+(1-\nu)/3)}}\|W(\an{e_j,e_k}) e_m\|_{L^2}^2 \\
& +  M^{-\nu} \sum_{j,l,m} \frac1{\lambda_j^{2(1+(1-\nu)/3)} \lambda_l^{2(1+(1-\nu)/3)} \lambda_m^{2(1+(1-\nu)/3)}} \|W(\an{e_m,e_l}) e_j\|_{L^2}^2 \\
& +  M^{-\nu} \sum_{k,l,m} \frac1{\lambda_k^{2(1+(1-\nu)/3)} \lambda_l^{2(1+(1-\nu)/3)} \lambda_m^{2(1+(1-\nu)/3)}} \|W(\an{e_l,e_m})e_k\|_{L^2}^2
.
\end{align*}
By symmetry of the sum, we obtain
\[
\sum_{(j,k,l,m)\in I_{M,N}}\frac{1}{\lambda_j^2\lambda_k^2\lambda_l^2\lambda_m^2} |A_{j,k,l,m}|^2 \leq  4 M^{-\nu} \sum_{j,k,l} \frac1{\lambda_j^{2(1+(1-\nu)/3)} \lambda_k^{2(1+(1-\nu)/3)} \lambda_l^{2(1+(1-\nu)/3)}}\|W(\an{e_k,e_j})e_l \|_{L^2}^2.
\]
We use that 
\[
 \|W(\an{e_k,e_j})e_l \|_{L^2} \leq \|W(\an{e_k,e_j})\|_{L^q} \|e_l\|_{L^p}.
\]
Using that $W$ is continuous from $L^{p/2}$ to $L^q$, and that $\|e_l\|_{L^p} \lesssim \lambda_l^{2(\frac{d}{2} - \frac{d}{p})} = \lambda_l^{2d/q}$ (see for instance \cite{szeg1939orthogonal} p391 exercise 91 with $x = \cos\theta$, with $\mu = \frac{d-2}{2}$ and $\alpha = \frac{d}{2}$) we get
\[
 \|W(\an{e_k,e_j})e_l \|_{L^2} \leq \|\an{e_k,e_j}\|_{L^p/2} \|e_l\|_{L^p} \leq \lambda_j^{2d/q}\lambda_k^{2d/q}\lambda_l^{2d/q}.
\]
We deduce
\[
\sum_{(j,k,l,m)\in I_{M,N}}\frac{1}{\lambda_j^2\lambda_k^2\lambda_l^2\lambda_m^2} |A_{j,k,l,m}|^2 \lesssim M^{-\nu} \Big(\sum_j \lambda_j^{2(-1-(1-\nu)/3 + 2d/q)}\Big)^3.
\]
We have 
\[
- \frac{1-\nu}{3} + \frac{2d}{q} < -\frac{1-(1-6d/q)}{3} + \frac{2d}{q} <0
\]
hence the series converges.
\end{proof}

For the rest of the paper, we fix $\nu \in (\frac{6d}{q}, 1 - \frac{6d}q)$. We notice that since $q>12d$, this interval is not empty.

\begin{corollary}\label{cor:CauchyLp}
    The sequence $(\mathcal{E}_N\circ P_N)_{N \in\N}$ is a Cauchy sequence in $L^r(\mu)$, for any $r\geq2$.

    Moreover for $M\leq N$ \begin{equation}\label{eq:cauchyboundlp}
        \|\mathcal{E}_N\circ P_N-\mathcal{E}_M\circ P_M\|_{L^r(\mu)}=(\E[\lvert G_N-G_M\rvert^r])^\frac{1}{r}\lesssim (r-1)^{2}\frac{1}{M^{\frac{\nu}{2}}}.
    \end{equation}
\end{corollary}

Corollary \ref{cor:CauchyLp} is an immediate consequence of Proposition \ref{prop;G_N Cauchy} and the following lemma, proven in the Chapter 5 of \cite{janson1997gaussian} (Hypercontractivity lemma).

\begin{lemma}
    Let $\overline{g}=(g_n)_n\in\N$ be a sequence of independant standard real-valued Gaussian random variables. Given $k\in \N$, if $(P_j)_{j\in\N}$ is a sequence of polynomials in $\overline{g}$ of degree at most $k$. Then, for any $r\geq2$,we have $$\Bigg\|\sum_{j\in \N}P_j(\overline{g})\Bigg\|_{L^p(\Omega)}\leq (r-1)^{k/2}\Bigg\|\sum_{j\in \N}P_j(\overline{g})\Bigg\|_{L^2(\Omega)}.$$
\end{lemma}

\begin{proposition}\label{prop:Nthordermeasure}
We set 
\begin{equation}\label{def:Nthordermeasure}
        R_N(\psi)=e^{-\mathcal{E}_N(P_N\psi)}.
\end{equation}
The sequence $(R_N)_N$ is bounded in $L^r(\mu)$ for any $r\in [1,\infty)$. Moreover, $R_N$ converges to some $R$ in $L^r(\mu)$ as $N\to+\infty$.
\end{proposition}

The first step is to prove a lesser bound on $\mathcal{E}_N$.

\begin{lemma}\label{lemma:nelsonestimates} 
    There exists $C_0>0$ such that \begin{equation}\label{eq:nelsonestimate}
        \mathcal{E}_N(\psi)\geq -C_0 N^{3d/q}.
    \end{equation}
\end{lemma}

\begin{proof}
By definition, we have
\begin{multline*}
 \mathcal{E}_N(\psi)=  \int_{\S^d} dxdy w(x,y) \Big[ \lvert \psi (x)\rvert^2\lvert \psi(y)\rvert^2-2\lvert \psi (x)\rvert^2\sigma_N(y) 
      - 2\overline{\psi (x)}\sigma_N(x,y)\psi\\
      +\sigma_N(x)\sigma_N(y)+Tr(\sigma_N(x,y)\sigma_N(y,x))\Big].
\end{multline*}
We have that
\begin{align*}
\overline{\psi (x)}\sigma_N(x,y)\psi = & \E(\an{\psi(x),\varphi_N(x)}\an{\varphi_N(y),\psi(y)}) \\
\leq  & |\psi(x)|\E(|\varphi_N(x)|\,|\varphi_N(y)|) |\psi(y)| \\
\leq & |\psi(x)|\sqrt{\sigma_N(y)} \sqrt{\sigma_N(x)} |\psi(y)| \\
\leq &\frac12\Big( |\psi(x)|^2 \sigma_N(y) + |\psi(y)|^2\sigma_N(x)\Big).
\end{align*}
We use that $w$ has non negative values to get
\[
 \mathcal{E}_N(\psi)\geq  \int dxdy w(x,y) \Big[ \lvert \psi (x)\rvert^2\lvert \psi(y)\rvert^2-4\lvert \psi (x)\rvert^2\sigma_N(y) 
+\sigma_N(x)\sigma_N(y)+Tr(\sigma_N(x,y)\sigma_N(y,x))\Big].
\]
We use that 
\[
(|\psi(x)|^2 - 2\sigma_N(x))(|\psi(y)|^2 - 2\sigma_N(y)) = |\psi(x)|^2|\psi(y)|^2- 2|\psi(x)|^2\sigma_N(y) -2|\psi(y)|^2\sigma_N(x) + 4\sigma_N(x) \sigma_N(y)
\]
to get
\[
 \mathcal{E}_N(\psi)\geq  \int dxdy w(x,y)\Big[ (|\psi(x)|^2 - 2\sigma_N(x))(|\psi(y)|^2 - 2\sigma_N(y)) 
-3\sigma_N(x)\sigma_N(y)+Tr(\sigma_N(x,y)\sigma_N(y,x))\Big].
\] 
Because $W$ is a non negative operator, we get
\[
 \mathcal{E}_N(\psi)\geq  \int dxdy w(x,y) \Big[
-3\sigma_N(x)\sigma_N(y)+Tr(\sigma_N(x,y)\sigma_N(y,x))\Big].
\] 
And again by non negativity of $w$ and the fact that $\sigma_N(y,x) = \sigma_N(x,y)^*$, we have
\[
 \mathcal{E}_N(\psi)\geq  \int dxdy w(x,y) \Big[
-3\sigma_N(x)\sigma_N(y)\Big].
\]
We deduce that
\[
 \mathcal{E}_N(\psi)\geq -3 \sum_{j ,k\leq N } \frac1{\lambda_j^2\lambda_k^2} \an{W(|e_j|^2) , |e_k(y)|^2}.
\]
We have by Cauchy--Schwarz that 
\[
 \an{W(|e_j|^2) , |e_k(y)|^2} \leq \|e_k\|_{L^2} \|W(|e_j|^2)e_k\|_{L^2} \leq \|e_k\|_{L^p} \|e_j\|_{L^p}^2 .
\]
We deduce 
\begin{align*}
 \mathcal{E}_N(\psi)\geq & -3 \sum_{j , k\leq N } \lambda_j^{-2 + 4d/q} \lambda_k^{-2 + 2d/q} \\
 \geq  & - C N^{3d/q}.
\end{align*}

\end{proof}

\begin{proof}[Proof of Proposition \ref{prop:Nthordermeasure}]
    We have that by using Lemma 4.5 in \cite{Tzvetkov2006ConstructionOA} and inequality \eqref{eq:cauchyboundlp} in Corollary \ref{cor:CauchyLp} for $M\geq N$, 
    \begin{equation}\label{eq:ineqmu}
        \mu(\lvert \mathcal{E}_M(P_M\psi)-\mathcal{E}_N(P_N\psi)\rvert>\lambda)\leq Ce^{-cN^{\frac{\nu}{4}}\lambda^\frac{1}{2}}.
    \end{equation}

    And we have 
 \[
 \|R_N\|_{L^r(\mu)}^p=\int_{H^s}e^{-r\mathcal{E}_N(P_N\psi)}d\mu(\psi)=\int_0^\infty \mu(e^{-r\mathcal{E}_N(P_N\psi)}>\alpha)d\alpha\leq 1+\int_1^\infty \mu(-r\mathcal{E}_N(P_N\psi)>\ln\ \alpha)d\alpha.
 \]
    
    We claim that under our hypothesis $\int_1^\infty \mu(e^{-p\mathcal{E}_N(P_N\psi)}>\alpha)d\alpha$ is finite. Indeed, we set $\lambda=\ln\alpha$ and $N_0:=\lfloor \Big(\frac{\lambda}{2rC_0}\Big)^{q/(3d)}\rfloor$ such that $\lambda\geq 2rC_0 N_0^{3d/q} $, with $C_0$ defined in Lemma \ref{lemma:nelsonestimates}. 
    
If $N\leq N_0$, we have that 
\[
-r\mathcal E_N(P_N\psi) \leq  rC_0 N^{3d/q} < \lambda
\]
and thus 
 \[
  \mu(-r\mathcal{E}_N(P_N\psi)>\lambda) = 0.
 \]
    
Let $N\geq N_0$, since   $-r\mathcal E_{N_0}(P_{N_0}\psi) \leq  rC_0 N_0^{3d/q} \leq \frac{\lambda}{2}$, we have 
\[
  \mu(-r\mathcal{E}_N(P_N\psi)>\lambda) \leq \mu (-r\mathcal{E}_N(P_N\psi) + r\mathcal E_{N_0}(P_{N_0}\psi) > \frac{\lambda}{2}).
\]
We deduce 
\[
  \mu(-r\mathcal{E}_N(P_N\psi)>\lambda) \leq C e^{-c N_0^{\nu/4} \lambda^{1/2}}.
\]
Using that as $\lambda$ goes to $\infty$, we have $N_0 \sim \lambda^{q/(3d)}$, we deduce 
\[
  \mu(-r\mathcal{E}_N(P_N\psi)>\lambda) \leq C e^{-c \lambda^{\nu q/(12d)} \lambda^{1/2}}.
\]
Since $\nu > \frac{6d}{q}$, we have that $\nu q/(12d) > 1/2$ and thus
\[
\| R_N\|_{L^r} \leq 1 + \int_1^{\infty} C e^{-c (\ln \alpha)^{1+}} d\alpha
\]
which is integrable. 
    
    Thus $R_N\in L^r(\mu)$ with a bound $C_r$ uniform in $N$. Since $(\mathcal{E}_N(P_N\psi))$ converges towards $\mathcal{E}(\psi)$ in $L^r$, we get that the convergence also holds in probability with respect to $\mu$.  Therefore, by composition of functions, $R_N(\psi)$ converges in probability to $R(\psi)=e^{-2\mathcal{E}(\psi)}$. And as a consequence  $R_N(\psi)$ converges to $R(\psi)=e^{-2\mathcal{E}(\psi)}$ in $L^p(\mu)$. 
    
    Indeed, let 
\[
A_{N,\varepsilon}=\{\lvert R_N(\psi)-R(\psi)\rvert\leq \varepsilon\},
\]
we have that $\mu(A^c_{N,\varepsilon})\to 0$ when $N\to+\infty$.

We then obtain that 
\[
\|R-R_N\|_{L^r(\mu)}\leq \|(R-R_N)\mathbf{1}_{A_{N,\varepsilon}}\|_{L^r(\mu)}+\|(R-R_N)\mathbf{1}_{A_{N,\varepsilon}^c}\|_{L^r(\mu)}.
\]

As $R$ and $R_N$ are in $L^{2r}$ with a bound $C_{2r}$ uniform in $N$, we have that 
\[
 \|R-R_N\|_{L^{2r}(\mu)}\leq 2C_{2r}.
 \]

By definition of $A_{N,\varepsilon}$, using Hölder's inequality one gets 
\[
\|R-R_N\|_{L^r(\mu)}\leq \varepsilon \mu(A_{N,\varepsilon})^\frac{1}{r}+2C_{2r}\mu(A_{N,\varepsilon}^c)\lesssim \varepsilon,
\]
for $N$ large enough. 
    
\end{proof}

\subsection{Renormalized nonlinearity}\label{subsec:renormalizednonlin}

In this section we work on the Nth ordered renormalized nonlinearity defined in \eqref{def:NthorderNL}. In particular, we prove the following result

\begin{proposition}\label{prop:truncnonlin} The sequence
    $(F_N)_N$ is a Cauchy sequence in $L^2(\mu,H^s(\S^d))$.
\end{proposition}

We begin with a lemma. 

\begin{lemma}
    We have the following equality 
\[
F_N(\varphi_N)=\sum_{0\leq j,k,l\leq N}\frac{1}{\lambda_j\lambda_k\lambda_l}\bigg( W(\langle e_j,e_k\rangle) e_l:\overline{g_j}g_kg_l:\bigg),
\]
where $:\overline{g_j}g_kg_l:=\overline{g_j}g_kg_l-\delta_j^kg_l-\delta_j^lg_k$.
\end{lemma}

\begin{proof}
    We have that 
\[
F_N(\varphi_N)=\int dxdy w(x,y) \Big( \lvert \varphi_N\rvert^2(y) \varphi_N(x)- \varphi_N(x)\sigma_N(y)-\sigma_N(x,y)\varphi_N(y)\Big).
\]
We compute term by term and we have 
\[
\lvert \varphi_N\rvert^2(y) \varphi_N(x)=\sum_{j,k,l\leq N}\frac{\overline{g_j}g_kg_l}{\lambda_j\lambda_k\lambda_l}\langle e_j(y),e_k(y)\rangle e_l(x) 
\]
and
\[
\varphi_N(x)\sigma_N(y)=\sum_{j,l\leq N}\frac{g_l}{\lambda_j^2\lambda_l}\lvert e_j(y)\rvert^2e_l(x) 
\]
and 
\[
\sigma_N(x,y)\varphi_N(y)=\sum_{j,k\leq N}\frac{g_k}{\lambda_j^2\lambda_k}\langle e_j(y),e_k(y)\rangle e_j(x).
\]
Adding the three terms gives the result.
\end{proof}

\begin{proof}[Proof of proposition \ref{prop:truncnonlin}]
    We have that $$\|\ \|F_N(\psi)\|_{H^s} \|_{L^2(\mu)}^2\leq \sum_{n\in \Z}\langle n\rangle^{2s}\|\langle F_N(\psi),e_n\rangle\|_{L^2(\mu)}^2.$$

     And one has 
     \[
     \|\langle F_N(\psi),e_n\rangle\|_{L^2(\mu)}^2=\E[\lvert\langle F_N(\varphi_N),e_n\rangle\vert^2]
     \]

     then by the previous lemma, one gets that 
     \[
     \|\langle F_N(\psi),e_n\rangle\|_{L^2(\mu)}^2=\sum_{0\leq k,j,l,k',j',l'\leq N}\frac{1}{\lambda_j\lambda_k
     \lambda_l} \frac{1}{\lambda_{j'}\lambda_{k'}\lambda_{l'}}A_{j,k,l,n}\overline{A_{j',k',l',n}}\E[:\overline{g_j}g_kg_l:\,:\overline{g_{k'}}\overline{g_{l'}}g_{j'}:],  
     \]
     where we recall that $A_{j,k,l,m}=\int W\langle e_j,e_k\rangle\langle e_l,e_m\rangle.$

Using a similar argument as in the proof of Proposition \ref{prop;G_N Cauchy}, we get that for $N\geq M$,
\[
\|\ \|F_N(\psi) - F_M(\psi)\|_{H^s} \|_{L^2(\mu)}^2\lesssim \sum_{n\in \N}\sum_{j,k,l\in I_{N,M}}\frac{1}{\lambda_j^2\lambda_k^2\lambda_l^2\lambda_n^{-4s}}\lvert A_{j,k,l,n}\rvert^2
\]
where 
\[
I_{N,M} = \{(j,k,l) \in [|0,N|]^3\; |\; \max(j,k,l) \geq M\}.
\]
We deduce that
\[
\|\ \|F_N(\psi) - F_M(\psi)\|_{H^s} \|_{L^2(\mu)}^2\lesssim \sum_{j,k,l,n \in J_{N,M}}\frac{1}{\lambda_j^2\lambda_k^2\lambda_l^2\lambda_n^{-4s}}\lvert A_{j,k,l,n}\rvert^2
\]
where
\[
J_{N,M} = \{(j,k,l,n) \in \N^4\; |\; \max(j,k,l,n) \geq M\}.
\]
As $s<-\frac{1}{2}$, using the same arguments as in the proof of Proposition \ref{prop;G_N Cauchy}, we get the result announced, with a bound 
\[
\|\ \|F_N(\psi)-F_M(\psi)\|_{H^s} \|_{L^2(\mu)}\lesssim \frac{1}{M^{\frac{\nu}{2}}}.
\]

\end{proof}

\section{Existence of a solution for the renormalized Dirac equation}

In this section we prove Theorem \ref{th:mainth}. We fix $s<\frac{1}{2}$. We begin by constructing global in time dynamics for the truncated Wick equation \begin{equation}\label{eq:trunceq}
    i\partial_t \psi=iD_d\psi+F_N(\psi).
\end{equation}

We set the truncated Gibbs measure $\rho_N$ as $d\rho_N=Z_N^{-1}R_N(\psi)d\mu$, where $Z_N$ is a normalizing constant. 

\begin{proposition}\label{prop:gwptrunc}
    The truncated Wick ordered equation \eqref{eq:trunceq} is globally well-posed in $H^s(\S^d)$. Moreover, $\rho_N$ is invariant under the flow of \eqref{eq:trunceq}.
\end{proposition}

\begin{proof}
    Let $\psi_N=P_N\psi$. Then \eqref{eq:trunceq} can be decomposed into the nonlinear evolution equation for $\psi_N$ on the low frequency part $\{n\leq N\}$ 
    \begin{equation}\label{eq:ODE1}
        i\partial_t \psi_N=iD_d\psi_N+F_N(\psi_N),
    \end{equation}

    and a linear ODE for each high frequency $n>N$\begin{equation}\label{eq:ODE2}
        i\partial_t (\psi-\psi_N)=iD_d (\psi-\psi_N).
    \end{equation}

    As a linear equation, any solution to \eqref{eq:ODE2} exists globally in time. By viewing \eqref{eq:ODE1} in the Fourier side, we see that \eqref{eq:ODE1} is a finite dimensional system of ODEs. The non-linearity $F_N$ being polynomial, it is Lipshitz-continuous. Hence by the Cauchy--Lipschitz theorem, we obtain local well-posedness of \eqref{eq:ODE1}.

    We recall that at fixed $N$, the potential energy $\mathcal E_N(u)$ is bounded by below. Since $H_N(u)$ is conserved by the flow, we deduce that as long as the solution $u$ is well-defined, the kinetic energy 
    \[
    \frac12 \an{u,iD_d u}_{L^2}
    \]
    is bounded. Since $u$ is restricted to the positive spectrum of $iD_d$, the kinetic energy is the square of a norm on $u$. Given that all norms are equivalent in finite dimension, this is sufficient to propagate the flow of \eqref{eq:ODE1} to all times and thus obtain \emph{global well-posedness} for \eqref{eq:ODE1}.
    
    Next we show that that $\rho_N$ is invariant under the flow of \eqref{eq:trunceq}. We write $\rho_N=\hat{\rho}_N\otimes\mu_N^\bot$, where $\mu_N^\bot$, is the Gaussian measure on the high frequencies $\{n>N\}$. The measure  $\mu_N^\bot$ is invariant under the flow of \eqref{eq:ODE2}. On the other hand, \eqref{eq:ODE1} is the finite dimensional Hamiltonian dynamics corresponding to $H^N_{Wick}(\psi_N)$ with 
 \[
H^N_{Wick}(\psi)=\frac{1}{2}\int_{\S^d}\overline{\psi}D\psi+\frac{1}{2}\int_{\S^d}\mathcal{E}_N(P_N\psi).
\]
 Thus, we that $\hat{\rho}_N$ is invariant under \eqref{eq:ODE1}. Therefore, $\rho_N$ is invariant under the flow of \eqref{eq:trunceq}.
\end{proof}

Following \cite{oh2018pedestrian}, we set the following definition

\begin{definition}
    Let $\Phi_N:H^s(\S^d)\to C(\R,H^s(\S^d))$ be the flow of \eqref{eq:trunceq} constructed in the proposition above. \\ For $t\in \R$, we denote $\Phi_N(t):H^s(\S^d)\to H^s(\S^d)$ the map defined by $(\Phi_N(t))(\psi)=(\Phi_N(\psi))(t)$. We endow $C(\R,H^s(\S^d))$ with the compact-open topology (that is to say the uniform convergence on all compact sets of $\R$). Under this topology, $\Phi_N$ is continuous from $H^s(\S^d)$ into $C(\R,H^s(\S^d))$. \\ We also denote 
\[
\nu_N=\rho_N\circ \Phi_N^{-1},
\]
that is to say that for any measurable function $F:C(\R,H^s(\S^d))\to \R$ we have 
\[
\int_{\C(\R,H^s(\S^d))}F(\psi)d\nu_N(\psi)=\int_{H^s(\S^d)}F(\Phi_N(\varphi))d\rho_N(\varphi).
\]
\end{definition}

We then show the following proposition 

\begin{proposition}\label{prop:sk+pr}
    There exists a subsequence $\nu_{N_j}$ that converges weakly to some probability mesure $\nu$ on $C(\R,H^s(\S^d))$. 

    Moreover, there exists another probability space $(\Tilde{\Omega},\Tilde{\mathcal F},\Tilde{P})$, a sequence $\psi_{N_j}$ of $C(\R,H^s(\S^d))$-valued random variables, and a $C(\R,H^s(\S^d))$-valued random variable $\psi$ such that \begin{itemize}
    \item the law of $\psi_{N_j}$ is $\nu_{N_j}$;
    \item the law of $\psi$ is $\nu$;
    \item the sequence $(\psi_{N_j})_j$ converges to $\psi$ in $C(\R,H^s(\S^d))$ almost surely with respect to $\Tilde{P}$.
    \end{itemize}
\end{proposition}

\begin{proof}
    The proof is similar to the one in \cite{oh2018pedestrian} (Part 5.2 and 5.3), we recall the main element of the proof. 

    The proof relies on Prokhorov's theorem (p59 in \cite{billingsley2013convergence}) and Skorokhod's theorem (p70 in \cite{billingsley2013convergence}). The objectives is then to show that the family of measures $(\nu_N)_N$ is tight in $\mathcal C(\R,H^s)$, that is to say that for every $\varepsilon>0$, there exists a compact set $K_\varepsilon$ such that $\nu_N(K_\varepsilon^c)\leq \varepsilon,$ for all $N\in\N$.

    \textbf{Step 1:} Let $r\geq 1$, there exists a constant $C_r>0$ such that 

    \begin{equation}\label{eq:flowLpt}
        \|\ \|\psi\|_{L_T^rH^s}\|_{L^r(\nu_N)}\leq C_rT^{\frac{1}{r}},
    \end{equation}
    and 
    \begin{equation}\label{eq:flowLpt2}
        \|\ \|\psi\|_{\dot W_T^{1,r}H^{s-1}}\|_{L^r(\nu_N)}\leq C_rT^{\frac{1}{r}}.
    \end{equation}

Indeed, by Fubini's theorem and the definition of $\nu_N$ 
\[
\|\ \|\psi\|_{L_T^rH^s}\|_{L^r(\nu_N)}=\|\ \|\Phi_N(t)(\psi)\|_{L_T^rH^s}\|_{L^r(\rho_N)}=\|\ \|\Phi_N(t)(\psi)\|_{L_T^r(\rho_N)H^s}\|_{L^r},
\]
thus by the invariance of $\rho_N$ and Hölder's inequality
\[
\|\ \|\psi\|_{L_T^rH^s}\|_{L^r(\nu_N)}=(2T)^\frac{1}{r}\|\psi\|_{L_T^r(\rho_N)H^s}\leq(2T)^\frac{1}{r}\|R_N\|_{L^{2r}(\mu)}\|\psi\|_{L^{2r}(\mu)H^s},
\]
which gives the first inequality by estimates on $R_N$ and Lemma 2.6 of \cite{oh2018pedestrian}. 

The second inequality can be obtained in the same way, using the definition of $\nu_N$. Indeed we have 
\[
\|\ \|\psi\|_{\dot W_T^{1,r}H^{s-1}}\|_{L^r(\nu_N)}\leq \|\ \|D\psi\|_{L_T^{r}H^{s-1}}\|_{L^r(\nu_N)}+\|\ \|F_N(\psi)\|_{L_T^{r}H^{s-1}}\|_{L^r(\nu_N)}.
\]

The first term is estimated by \eqref{eq:flowLpt} and, reasoning as in the previous proof of \eqref{eq:flowLpt} one gets   
\[
\|\ \|F_N(\psi)\|_{L_T^{r}H^{s-1}}\|_{L^r(\nu_N)}\lesssim (2T)^\frac{1}{r}\|R_N\|_{L^{2r}(\mu)}\|F_N(\psi)\|_{L^{2r}(\mu)H^{s-1}}.
\]
Using Proposition \ref{prop:truncnonlin} gives \eqref{eq:flowLpt2}.

    \textbf{Second step:} We set $s<s_1<s_2<-\frac{1}{2}$ and for $\alpha\in(0,1)$, we consider $C_T^\alpha H^{s_1}=C^\alpha([-T,T],H^{s_1}(\S^d))$ defined by the norm 
\[
\|\psi\|_{\C_T^\alpha H^{s_1}}=\underset{t_1\neq t_2\in [-T,T]}{\sup}\frac{\|\psi(t_1)-\psi(t_2)\|_{H^{s_1}}}{\lvert t_1-t_2\rvert^\alpha}+\|\psi\|_{L_T^\infty H^{s_1}}.
\]
By Arzela Ascoli's theorem, the embedding $C_T^\alpha H^{s_1}\subset C_T H^{s} $ is compact. 

By Lemma 3.3 of \cite{burq2018remarks} (interpolation lemma), Sobolev's inequality and the first step, one obtains that for $r$ large enough 
\[
\|\|\psi\|_{\C_T^\alpha H^{s_1}}\|_{L^r(\nu_N)}\leq C_rT^\frac{1}{r}.
\]

We then define for $j\in\N $ the number $T_j=2^j$. We set for $\varepsilon \in (0,1)$,
\[
K_\varepsilon=\{ \psi\in C(\R,H^s(\S^d)); \|\psi\|_{\C_{T_j}^\alpha H^{s_1}}\leq c_0\varepsilon^{-1}T_j^{1+\frac{1}{p}} \text{ for all }j\in \N \}.
\]
By Markov's inequality and choosing $c_0$ large enough one gets 
    \[
    \nu_N(K_\varepsilon^c)<\varepsilon.
    \]
    Indeed, we have 
    \[
    K_\varepsilon^c = \bigcup_j \{ \psi \;|\; \| \psi\|_{\C_{T_j}^\alpha H^{s_1}} > c_0 \varepsilon^{-1}T_j^{1+1/r}\}.
    \]
    We deduce 
    \[
    \nu_N(K_\varepsilon^c) \leq \sum_j \nu_N(\{ \psi \;|\; \| \psi\|_{\C_{T_j}^\alpha H^{s_1}} > c_0 \varepsilon^{-1}T_j^{1+1/r}\}).
    \]
    We have 
    \[
    \nu_N(\{ \psi \;|\; \| \psi\|_{\C_{T_j}^\alpha H^{s_1}} > c_0 \varepsilon^{-1}T_j^{1+1/r}\}) = \nu_N(\{ \psi \;|\; \| \psi\|_{\C_{T_j}^\alpha H^{s_1}}^r > c_0^r \varepsilon^{-r}T_j^{r+1}\}).
    \]
    By Markov's inequality, we get
    \[
    \nu_N(\{ \psi \;|\; \| \psi\|_{\C_{T_j}^\alpha H^{s_1}} > c_0 \varepsilon^{-1}T_j^{1+1/r}\}) \leq c_0^{-r} \varepsilon^{r}T_j^{-r-1}  \|\,\| \psi\|_{\C_{T_j}^\alpha H^{s_1}}\|_{L^r(\nu_N)}^r.
    \]
    This yields 
    \[
    \nu_N(\{ \psi \;|\; \| \psi\|_{\C_{T_j}^\alpha H^{s_1}} > c_0 \varepsilon^{-1}T_j^{1+1/r}\}) \leq c_0^{-r}C_r^r \varepsilon^{r}T_j^{-r} .
    \]
    We sum on $j$ and get 
    \[
    \nu_N(K_\varepsilon^c) \leq 2 c_{0}^{-r}C_r^r \varepsilon.
    \]
    Taking $c_0 \geq  C_r 2^{1/r}$, we get the result.

    It remains to prove that $K_\varepsilon$ is compact in $C(\R,H^s(\S^d))$. Let $(u_n)_n$ be a sequence of $K_\varepsilon$. By definition of $K_\varepsilon$, the sequence $(u_n)$ is bounded in $C_{T_j}^\alpha H^{s_1}$ for any $j\in \N$. By a diagonal argument, there exists a subsequence $(u_{n_l})$ that converges in $C_{T_j}^\alpha H^{s_1}$ for any $j\in \N$. In particular, the sequence converges uniformly in $H^s$ on any compact time interval. Thus it converges in $C(\R,H^s(\S^d))$ and $K_\varepsilon$ is compact in $C(\R,H^s(\S^d))$.

    \textbf{Conclusion: }We showed that $(\nu_N)$ is a tight family of measures, thus, by Prokhorov's theorem it exists a subsequence $(\nu_{N_j})$ that converges weakly to some probability measure $\nu$ on $C(\R,H^s(\S^d))$. The Skorokhod's theorem allows us to conclude the proof of the theorem by giving the existence of the sequence $\psi_{N_j}$ and $\psi$. 
    
\end{proof}

We now prove the following proposition, which will conclude the proof of Theorem \ref{th:mainth}

\begin{proposition}
    Let $\psi_{N_j}$ and $\psi$ be as above. Then we have for all $t\in \R$, \begin{itemize}
        \item the law of $\psi_{N_j}(t)$ is $\rho_{N_j}$;
        \item the law of $\psi(t)$ is $\rho_\infty:=Z^{-1}R(\phi)d\mu$;
    \item the functions $\psi_{N_j}$ are solutions to the truncated Wick ordered equations \eqref{eq:trunceq} and $\psi$ is a solution in the distributional sense of \begin{equation}\label{eq:Wickeq}
        i\partial_t \psi =iD_d \psi + F(\psi). 
    \end{equation}
    \end{itemize}
\end{proposition}

\begin{proof}
    For the first part of the proposition, we fix some $t\in\R$, and $R_t:C(\R;H^s(\S^d))\to H^s(\S^d)$ defined by $R_t(\varphi)=\varphi(t)$. Then, $R_t$ is continuous and we have by Proposition \ref{prop:gwptrunc} the law of $\psi_{N_j}(t)$ is $\nu_{N_j}\circ R_t^{-1} = \rho_{N_j}$.

By Proposition \ref{prop:sk+pr}, $\psi_{N_j}(t)$ converges to $\psi(t)$ in $H^s(\S^d)$ almost surely. So by dominated convergence theorem the $\psi_{N_j}(t)$ converges in law to $\psi(t)$. We deduce that the law of $\psi(t)$ is the limit of the sequence $(\rho_{N_j})_j$, that is $\rho_\infty$.

    For the second part of the proof, we set $X_j$ the $D'_{t,x}$-valued random variable (where $D'_{t,x}:=D'(\R\times \S^d)$ defined by 
\[
X_j=i\partial_t \psi_{N_j}-iD_dD\psi_{N_j}-F_N(\psi_{N_j}).
\] 
As $\psi_{N_j}$ is of law $\nu_{N_j}$ and by definition of $\nu_{N_j}$, the variable $X_j$ is almost surely $0$. In other words, $\psi_{N_j} $ is a solution to \eqref{eq:trunceq}, in the sense of distribution.

By almost sure convergence of $\psi_{N_j}$ in $C(\R,H^s(\S^d))$ we have 
\[
i\partial_t\psi_{N_j}-D\psi_{N_j}\to i\partial_t\psi-D\psi 
\]
in $D'_{t,x}$ as $j\mapsto\infty$ almost surely with respect to $\Tilde{P}$. 

We conclude by showing the almost sure convergence of $F_{N_j}(\psi_{N_j})$ to $F(\psi)$. To simplify we denote $F_j=F_{N_j}$ and $u_j=\psi_{N_j}$. We have 
\[
F_j(u_j)-F(\psi)=F(u_j) - F(\psi) + F_j(u_j) - F(u_j)
\]
We have that
\[
\|F_j(u_j) - F(u_j)\|_{L^2_{\tilde P},L^2_T,H^s} = \|(1-\Delta)^{s/2} (F_j(u_j) - F(u_j))\|_{L^2_{t,x},L^2_{\tilde P}}
\]
We use the definition of $u_j$ to get 
\[
\|F_j(u_j) - F(u_j)\|_{L^2_{\tilde P},L^2_T,H^s} = \|(1-\Delta)^{s/2} (F_j(X) - F(X))\|_{L^2_{t,x},L^2_{\nu_{N_j}}}
\]
We use that $\rho_{N_j}$ is invariant under the flow $\Phi_{N_j}$ to get
\[
\|F_j(u_j) - F(u_j)\|_{L^2_{\tilde P},L^2_T,H^s} = \sqrt T \|(1-\Delta)^{s/2} (F_j(X) - F(X))\|_{L^2_{x},L^2_{\rho_{N_j}}}
\]
We use the definition of $\rho_{N_j}$ to get
\[
\|F_j(u_j) - F(u_j)\|_{L^2_{\tilde P},L^2_T,H^s} \leq  \sqrt T \|R_{N_j}\|_{L^2_\mu} \|(1-\Delta)^{s/2} (F_j(X) - F(X))\|_{L^2_{x},L^4_{\mu}}
\]
that is to say
\[
\|F_j(u_j) - F(u_j)\|_{L^2_{\tilde P},L^2_T,H^s} \leq  \sqrt T \|R_{N_j}\|_{L^2_\mu} \|F_j(X) - F(X)\|_{L^4_{\mu},H^s}
\]
We use the boundedness of $(R_N)_N$, the convergence of $(F_N)_N$ along with the hypercontractivity corolary \ref{cor:CauchyLp} to conclude that 
\[
\|F_j(u_j) - F(u_j)\|_{L^2_{\tilde P},L^2_T,H^s} \rightarrow 0.
\]

Let $M>0$. In the same way as before, we have that 
\begin{align*}
\|F_M(u_j) - F(u_j)\|_{L^2_{\tilde P},L^2_T,H^s} \leq &  \sqrt T \|R_{N_j}\|_{L^2_\mu} \|F_M(X) - F(X)\|_{L^4_{\mu},H^s} \\
\|F_M(\psi) - F(\psi)\|_{L^2_{\tilde P},L^2_T,H^s} \leq & \sqrt T \|R\|_{L^2_\mu} \|F_M(X) - F(X)\|_{L^4_{\mu},H^s}.
\end{align*}
Since $(R_N)_N$ is bounded, we get that
\begin{align*}
\|F_M(u_j) - F(u_j)\|_{L^2_{\tilde P},L^2_T,H^s} \leq & C \sqrt T \|F_M(X) - F(X)\|_{L^4_{\mu},H^s} \\
\|F_M(\psi) - F(\psi)\|_{L^2_{\tilde P},L^2_T,H^s} \leq  & C \sqrt T  \|F_M(X) - F(X)\|_{L^4_{\mu},H^s}.
\end{align*}
Let $\varepsilon >0$ and pick $M$ such that $C \sqrt T \|F_M(X) - F(X)\|_{L^2_{\mu},H^s} \leq \varepsilon$, we have that
\[
\|F(u_j) - F(\psi)\|_{L^2_{\tilde P},L^2_T,H^s}  \leq 2\varepsilon + \|F_M(u_j) - F_M(\psi)\|_{L^2_{\tilde P},L^2_T,H^s}.
\]
We have that there exists a constant $C_M$ such that 
\[
\| F_M(X) - F_M(Y) \|_{H^s} \leq C_M (1+ \sup |P_M X| + \sup |P_M Y|) \sup |P_M X- P_M Y|.
\]
Because $P_M$ projects in finite dimension where all norms are equivalent, we get
\[
\| F_M(X) - F_M(Y) \|_{H^s} \leq C_M (1+ \|X\|_{H^s} + \|Y\|_{H^s})  \| X-  Y\|_{H^s}.
\]
We deduce that
\[
\|F_M(u_j) - F_M(\psi)\|_{L^2_{\tilde P},L^2_T,H^s} \leq C_M (1+ \|u_j\|_{L^4_{\tilde P},L^2_T,H^s} + \|\psi\|_{L^4_{\tilde P},L^2_T,H^s})  \| u_j -  \psi \|_{L^4_{\tilde P},L^2_T,H^s}.
\]
Fatou's lemma yields that $\psi \in L^4_{\tilde P},L^2_T,H^s$. Almost sure convergence yields convergence in probability. As we have already seen, convergence in probability along with boundedness in $L^4$ yields convergence in $L^4$. We deduce that 
\[
\|F(u_j) - F(\psi)\|_{L^2_{\tilde P},L^2_T,H^s} 
\]
converges to $0$ as $j$ goes to $\infty$.

    Then, up to a subsequence, $F_j(u_j)$ converges to $F(\psi)$ in $L^2_TH^s$ almost surely with respect to $\Tilde{P}$. We apply this to $T_l=2^l$ for any $l\in\N$. Then, by a diagonal argument, up to a subsequence, $F_j(u_j)$ converges to $F(\psi)$ in $L^2_{loc}H^s$ almost surely with respect to $\Tilde{P}$. In particular, up to a subsequence, $F_j(u_j)$ converges to $F(\psi)$ in $D'_{t,x}$ almost surely with respect to $\Tilde{P}$, and $\psi$ is a global-in-time solution to \eqref{eq:Wickeq}.
\end{proof}

\newpage

\bibliographystyle{amsplain}
\bibliography{biblio}

\end{document}